\documentclass{amsart}

\usepackage{amssymb, amsmath, amscd, amsthm}
\usepackage{lineno}
\usepackage{amsfonts}
\usepackage{psfrag}
\usepackage[all]{xy}
\usepackage{graphicx}

\newtheorem{theorem}{Theorem}[section]
\newtheorem{corollary}[theorem]{Corollary}
\newtheorem{lemma}[theorem]{Lemma}

\theoremstyle{definition}

\newtheorem{definition}[theorem]{Definition}
\theoremstyle{remark}
\newtheorem{remark}[theorem]{Remark}

\newtheorem{example}[theorem]{Example}

\newcommand{\p}{\varphi}
\newcommand{\R}{\mathbb{R}}

\newcommand{\fr}{\vec \varphi}

\numberwithin{equation}{section}

\begin{document}

\title[Stable comparison of multidimensional persistent homology groups]{Stable comparison of multidimensional persistent homology groups with torsion}

\author[P. Frosini]{Patrizio Frosini}
\address{Patrizio Frosini, ARCES, Universit\`a di Bologna,
via Toffano $2/2$, I-$40135$ Bologna, Italia\newline Dipartimento
di Matematica, Universit\`a di Bologna, P.zza di Porta S. Donato
5, I-$40126$ Bologna, Italia, tel. +39-051-2094478, fax.
+39-051-2094490} \email{frosini@dm.unibo.it}


\subjclass[2000]{}

\date{\today} 

\dedicatory{This paper is dedicated to the memory of Padre Armando Felice Verde. G.g.}



\begin{abstract}
The present lack of a stable method to compare persistent homology groups with torsion is a relevant problem in current research about Persistent Homology and its applications in Pattern Recognition. In this paper we introduce a pseudo-distance $d_T$ that represents a possible solution to this problem. Indeed, $d_T$ is a pseudo-distance between multidimensional persistent homology groups with coefficients in an Abelian group, hence possibly having torsion. Our main theorem proves the stability of the new pseudo-distance with respect to the change of the filtering function, expressed both with respect to the max-norm and to the natural pseudo-distance between topological spaces endowed with $\R^n$-valued filtering functions. Furthermore, we prove a result showing the relationship between $d_T$ and the matching distance in the $1$-dimensional case, when the homology coefficients are taken in a field and hence the comparison can be made.
\end{abstract}

\maketitle

\section*{Introduction}
The extension of the theory of Persistent Homology to the case of persistent homology groups with torsion is an important open problem. Indeed, the lack of a complete representation similar to the one given by persistent diagrams is a relevant obstacle to this development. However,  this fact does not prevent us from using persistent homology groups with torsion for shape comparison. In order to show this, we introduce here a new pseudo-distance $d_T$ between persistent homology groups with coefficients in an Abelian group (hence possibly having torsion), proving that $d_T$ is stable with respect to changes of the filtering functions, measured by the max-norm. We also express this stability as a lower bound for the natural pseudo-distance between topological spaces endowed with $\R^n$-valued filtering functions. This approach opens the way to the use of persistent homology groups in concrete applications, extending some ideas developed in \cite{FrMu99} for size homotopy groups to persistent homology.

Now, let us illustrate how persistent homology groups with torsion can appear in applications. As an example, let us imagine to have to study all the possible ways of grasping an object by a robotic hand endowed with $m$ fingers. Obviously, not all grasps can be considered equivalent. For example, some of them can be difficult to realize in practice, due to the position of the object and the technical capability of the hand. Therefore, it is reasonable to associate each grasp with a ``cost'', represented by a real number. In one of the simplest cases, we can consider two fingers grasping an object whose surface $\mathcal{S}$ is diffeomorphic to a spherical surface, in presence of friction. Each grasp can be seen as an unordered pair $\{p,q\}$ of points on $\mathcal{S}$, and we can consider the set $\mathcal{G}$ of all possible grasps (endowed with the Hausdorff distance between unordered pairs). If the cost of the grasp is given by a value $\varphi_\mathcal{G}(\{p,q\})$, we are naturally led to the problem of comparing the similarity of the grasps possible for two different objects with surfaces $\mathcal{S}$ and $\mathcal{S}'$, with respect to the cost functions $\varphi_\mathcal{G}:\mathcal{G}\to \R$, $\varphi_{\mathcal{G}'}:\mathcal{G}'\to \R$. This can be done by computing the (singular) persistent homology groups of each topological space $\mathcal{G}$ with respect to the filtering function $\varphi_\mathcal{G}$. In other words, we can consider the persistent homology groups ${H}_k^{(\mathcal{G},\p_\mathcal{G})}(u,v)$ given by the equivalence classes of cycles in ${H}_k(\mathcal{G}_v)$ that contain at least one cycle in $\mathcal{G}_u$, where $\mathcal{G}_t$ represents the set of points of $\mathcal{G}$ at which $\p_\mathcal{G}$ takes a value not greater than $t$. It is interesting to note that, in general, this procedure leads to groups with torsion, when the homology coefficients are the integer numbers.

To understand this, let us consider a particularly simple example.
Let us assume that the cost of the grasp is given by the value $\varphi_\mathcal{G}(\{p,q\})=-\|p-q\|$, so relating the cost of a grasp to the level of closure of the robotic hand. For the sake of simplicity, let us examine the case $\mathcal{S}=S^2=\{(x,y,z)\in\R^3:x^2+y^2+z^2=1\}$. In this case, the group ${H}_1^{(\mathcal{G},\p_\mathcal{G})}(u,v)$ has torsion, when $u<v<0$ and $u$ is close enough to zero. Indeed, we can consider the continuous function $f$ taking each unordered pair $\{p,q\}\in \mathcal{G}_v$ to the direction of the line $l$ through $p$ and $q$ (note that $\{p,q\}\in \mathcal{G}_v$ implies $p\neq q$, so that $l$ is well-defined). This direction can be seen as a point in the real projective plane $\mathbb{RP}^2$. Let us observe that, for $v$ close enough to zero, the function $f:\mathcal{G}_v\to \mathbb{RP}^2$ is a homotopy equivalence (the homotopy inverse of $f$ is the function $g:\mathbb{RP}^2\to \mathcal{G}_v$ that takes each direction $w$ to the unordered pair $\{p,q\}$, where $p$ and $q$ are the points at which the line through $(0,0,0)$, having direction $w$, meets $S^2$). Hence $f$ induces an isomorphism between the singular homology groups of $\mathcal{G}_v$ and $\mathbb{RP}^2$. As a consequence, since $H_1(\mathbb{RP}^2)$ has torsion, the persistent homology group ${H}_1^{(\mathcal{G},\p_\mathcal{G})}(u,v)$ has torsion, too, when $u$ is close enough to zero.
For an introduction to Topological Robotics we refer the interested reader to \cite{Fa08}.

\subsection*{Multidimensional persistence}
Persistent homology has turned out to
be a key mathematical method for studying the topology of data,
with applications in an increasing number of fields, ranging from
shape description (e.g., \cite{CaZo*05,CeFeGi06,MoSaSa08,VeUrFrFe93}) to data
simplification \cite{EdLeZo02} and hole detection in sensor networks \cite{DeGh07}. Recent surveys on the topic include \cite{EdHa08,EdHa09,Gh08,Zo05}.
Persistent homology describes topological events occurring through
the filtration of a topological space $X$ (e.g., creation, merging, cancellation of
connected components, tunnels, voids). Filtrations are usually
expressed by real functions $\p:X\to\R$ called {\em filtering
functions}. The main idea underlying this approach is that the
most important piece of information enclosed in geometrical data
is usually the one that is ``\emph{persistent}'' with respect to
the defining parameters.

Until recently, research on persistence has mainly focused on the use of scalar functions for describing filtrations.
The extent to which this theory can be generalized to a situation
in which two or more functions characterize the data is currently
under investigation \cite{BiCeFrGiLa08,CaDiFe10,CaLa11,CaSiZo09,CaZo09}. This
generalization to vector-valued functions is usually known as the
\emph{Multidimensional Persistence Theory}, where the adjective
multidimensional refers to the fact that filtering functions are
vector-valued, and has no connections with the dimensionality of
the space under study. The use of vector-valued filtering
functions in this context
enables the analysis of richer
data structures.

An important topic in current research about
multidimensional persistent homology is the \emph{stability
problem}. In plain words, we need to determine how the computation
of invariants in this theory is affected by the unavoidable
presence of noise and approximation errors. Indeed, it is clear
that any data acquisition is subject to perturbations and, if
persistent homology were not stable, then distinct computational
investigations of the same object could produce completely
different results. Obviously, this would make it impossible to use
such a mathematical theory in real applications.

Up to our knowledge, no theoretical result is available for the metric comparison of persistent homology groups in presence of torsion, at the time we are writing.

\subsection*{Prior works}
The problem of stability in persistent homology has been studied
by Cohen-Steiner, Edelsbrunner and Harer in \cite{CoEdHa07} for
scalar filtering functions. By using a descriptor called a
\emph{persistence diagram}, they prove that persistent Betti
numbers are stable under perturbations of filtering functions with
respect to the max-norm, provided that the considered filtering
functions are \emph{tame}. The same problem is studied in
\cite{CoEdHaMi10} for tame Lipschitz functions. In \cite{ChCo*09},
Chazal {\em et al.} use the concept of \emph{persistence
module} and prove stability under the assumption that it is
finite-dimensional. The problem of stability for scalar filtering
functions is also approached in \cite{dAFrLa10}, where it is solved
by assuming that the considered filtering functions are no more
than continuous, but only for the 0th homology.

Multidimensional persistence was firstly investigated in
\cite{FrMu99} as regards homotopy groups, and by Carlsson and
Zomorodian in \cite{CaZo07} as regards homology modules. In this
context, the first stability result has been obtained for the 0th
homology in \cite{BiCeFrGiLa08}: A distance between the 0th persistent
Betti numbers, also called \emph{size functions}, has been
introduced and proven to be  stable under perturbations of
continuous vector-valued filtering functions. Such a result has
been partially extended in \cite{CaDiFe10} for all homology
degrees, under the restrictive assumption that the vector-valued
filtering functions are \emph{max-tame}.

\subsection*{Contributions}
In this paper we introduce a new pseudo-distance $d_T$ between persistent homology groups, applicable also in the case that these groups have torsion (Theorem~\ref{isapseudodistance}). Moreover, we prove that $d_T$ is stable with respect to perturbation of the filtering function, measured by the max-norm (Corollary~\ref{cor2}). This result is a consequence of a new lower bound for the natural pseudo-distance $\delta$ (Theorem~\ref{bound}), which also implies an easily computable lower bound for $\delta$ (Corollary~\ref{cor1}).
Furthermore, we investigate the link between $d_T$ and the matching distance $d_{match}$ between persistent diagrams, in the particular case that the filtering functions take values in $\R$ and the coefficients are taken in a field, so that also $d_{match}$ can be applied (Theorem~\ref{dtlessthandelta}).

The structure of the paper is as follows. In Section~\ref{basics} we recall the definition of persistent homology group and show some examples.
In Section~\ref{results} the definition of $d_T$ is given and our results are proven.

\section{Some definitions and properties in multidimensional persistence}
\label{basics}

In this paper,
the following relations $\preceq$ and $\prec$ are defined in
$\R^n$: for $\vec u=(u_1,\dots,u_n)$ and $\vec v=(v_1,\dots,v_n)$,
we say $\vec u\preceq\vec v$ (resp. $\vec u\prec\vec v$) if and
only if $u_i\leq\ v_i$ (resp. $u_i<v_i$) for every index
$i=1,\dots,n$. Moreover, $\R^n$ is endowed with the usual
$\max$-norm: $\left\|(u_1,u_2,\dots,u_n)\right\|_{\infty}=\max_{1\leq i\leq
n}|u_i|$.

We shall use the following notations: $\Delta^+$ will be the open
set $\{(\vec u,\vec v)\in\R^n\times\R^n:\vec u\prec\vec v\}$.
For every $n$-tuple $\vec
u=(u_1,\dots,u_n)\in\R^n$ and for every function
$\vec\p:X\to\R^n$, we shall denote by $X\langle\fr\preceq \vec
u\,\rangle$ the set $\{x\in X:\varphi_i(x)\leq u_i,\
i=1,\dots,n\}$. If $X$ is a compact topological space and $\vec\p:X\to\R^n$ is a continuous function, we shall set $\|\vec \varphi\|_\infty=\max_{x\in X}\|\vec \varphi(x)\|_\infty$.

Now we can recall the definition of multidimensional persistent homology group.

\begin{definition}[Persistent homology group]\label{PersHom}
Let $k\in\mathbb{Z}$. Let $X$ be a topological space, and $\fr:X\to\R^n$ a continuous
function. Let $i_*:{H}_k(X\langle\vec\p\preceq\vec
u\rangle)\rightarrow {H}_k(X\langle\vec\p\preceq\vec
v\rangle)$ be the homomorphism induced by the inclusion map
$i:X\langle\vec\p\preceq\vec
u\rangle\hookrightarrow X\langle\vec\p\preceq\vec v\rangle$ with
$\vec u\preceq\vec v$, where ${H}_k$ denotes the $k$th
homology group with coefficients in an Abelian group $G$. If $\vec u\prec\vec v$, the image of
$i_*$ is called the {\em multidimensional
$k$th persistent homology group of $(X,\vec\p)$ at $(\vec u, \vec
v)$}. We shall denote it by the symbol ${H}_k^{(X,\vec\p)}(\vec u, \vec v)$.
The function ${H}_k^{(X,\vec\p)}$ is called the {\em multidimensional
$k$th persistent homology group of $(X,\vec\p)$}.
\end{definition}

In other words, the group ${H}_k^{(X,\vec\p)}(\vec u, \vec v)$ contains all and only the homology classes of
cycles born before or at $\vec u$ and still alive at $\vec v$. We observe that, up to Definition~\ref{PersHom}, each persistent homology group is a function ${H}_k^{(X,\vec\p)}$, taking each pair $(\vec u,\vec v)\in \Delta^+$ to a group.

In the rest of this paper, the superscript arrow will be omitted when the filtering function takes values in $\R$.

\begin{remark}\label{rem_hom}
The choice of the particular homology we use is not essential in this context, in the sense that the results proven in this paper do not depend on this choice. However, while we will not further discuss this point, we recall that some important properties of persistent homology groups depend on the homology that is considered. As an example, the use of \v{C}ech homology with real coefficients allows us to assume that the function $\beta_\varphi:\Delta^+\to \mathbb{N}\cup \{\infty\}$ that takes each pair $(\vec u,\vec v)\in \Delta^+$ to the rank of the multidimensional $k$th persistent homology group of $(X,\vec\p)$ at $(\vec u, \vec v)$, is right-continuous in both its variables (cf. \cite{CeDiFeFrLa09}).
\end{remark}

\begin{example}\label{ex_groups} Let $\mathbb{RP}^2$ be the real projective plane, represented by the unordered pairs $\{p,-p\}$ of opposite points of $S^2=\{(x,y,z)\in\R^3:x^2+y^2+z^2=1\}$. Let us consider the filtering functions $\varphi,\psi:\mathbb{RP}^2\to\R$, defined by setting
$\varphi(\{(x,y,z),-(x,y,z)\})=|z|$ and $\psi(\{(x,y,z),-(x,y,z)\})=2|z|$. Let us also consider the filtering function $\chi:S^2\to\R$, defined by setting $\chi(x,y,z)=|z|$. For each $(u,v)\in \Delta^+$, the persistent homology groups in degree $1$ computed at $(u,v)$ of $\varphi$, $\chi$ and $\psi$ are displayed in Figure~\ref{groups}, from left to right. The coefficients of the homology are taken in $\mathbb{Z}$.
\end{example}

\begin{figure}[h]
\begin{center}
\includegraphics[width=14cm]{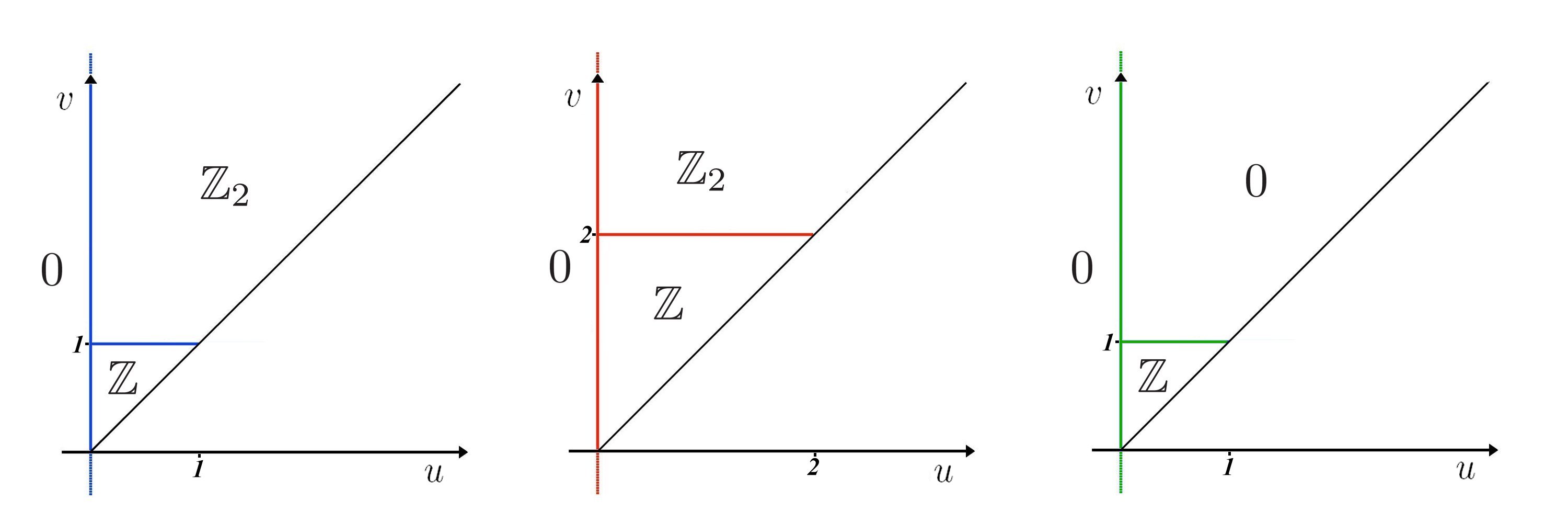}
\caption{Three examples of persistent homology groups with integer coefficients, in degree $1$. In the first and second case the topological space is the projective plane, represented by the unordered pairs $\{p,-p\}$ of opposite points of $S^2$, while in the third case the topological space is $S^2$. The filtering functions are $\varphi(\{(x,y,z),-(x,y,z)\})=|z|$, $\psi(\{(x,y,z),-(x,y,z)\})=2|z|$ and $\chi(x,y,z)=|z|$, respectively. In each region, the persistent homology group at the points of that region is displayed. The symbol $0$ denotes the trivial group.} \label{groups}
\end{center}
\end{figure}

Since we do not take coefficients in a field, the homology groups we consider have torsion, in the general case. As a consequence, they cannot be described by persistence diagrams (cf. \cite{CoEdHa07}). It follows that the classical matching distance cannot be applied, if we are interested in the information given by the torsion part of our groups. However, we can define a pseudo-distance that does not suffer from this limitation. We shall do this in the next section.

\section{The pseudo-distance $d_T$ and the proof of its stability}
\label{results}

In the following, $Ab$ will represent the collection of all Abelian groups, considered up to isomorphisms.

\begin{definition}\label{distancedef}
Let $X$, $Y$ be two topological spaces, and $\vec\varphi:X\to\R^n$, $\vec\psi:Y\to\R^n$ two continuous filtering functions. Let $ H^{(X,\vec\varphi)}_k:\Delta^+\to Ab$ and $ H^{(Y,\vec\psi)}_k:\Delta^+\to Ab$ the multidimensional persistent $k$-homology groups associated with the pairs $(X,\vec\varphi)$ and $(Y,\vec\psi)$, respectively. We assume that the homology coefficients are taken in an Abelian group. Let us consider the set $E$ of all $\epsilon\ge 0$ such that, setting $\vec\epsilon=(\epsilon,\ldots,\epsilon)\in \mathbb{R}^n$, the following statements hold for each $(\vec u,\vec v)\in \Delta^+$:
\begin{enumerate}
\item a surjective homomorphism from a subgroup of $ H^{(X,\vec\varphi)}_k(\vec u,\vec v)$ onto $ H^{(Y,\vec\psi)}_k(\vec u-\vec \epsilon,\vec v+\vec \epsilon)\}$ exists;
\item a surjective homomorphism from a subgroup of $ H^{(Y,\vec\psi)}_k(\vec u,\vec v)$ onto $ H^{(X,\vec\varphi)}_k(\vec u-\vec \epsilon,\vec v+\vec \epsilon)\}$ exists.
\end{enumerate}
We define $d_T\left( H^{(X,\vec\varphi)}_k, H^{(Y,\vec\psi)}_k\right)$ equal to $\inf E$ if $E$ is not empty, and equal to $\infty$ otherwise.
\end{definition}

Before proceeding, we recall that the term ``pseudo-distance'' means that the considered function verifies all the properties of a distance, with the possible exception of the one assuring that two elements having a vanishing distance must coincide. The term ``extended'' refers to the possibility that the function takes the value $\infty$.

\begin{theorem}\label{isapseudodistance}
The function $d_T$ is an extended pseudo-distance.
\end{theorem}

\begin{proof}
The following statements hold:
\begin{description}
\item[$i)$] $d_T$ is non-negative by definition;
\item[$ii)$]The equality $d_T\left( H^{(X,\vec\varphi)}_k, H^{(X,\vec\varphi)}_k\right)=0$ is obtained by taking $\epsilon=0$ and considering the identical homomorphism $id: H^{(X,\vec\varphi)}_k(\vec u,\vec v)\to H^{(X,\vec\varphi)}_k(\vec u-\vec \epsilon,\vec v+\vec \epsilon)$;
\item[$iii)$]$d_T$ is symmetrical by definition;
\item[$iv)$]Assume that $d_T\left( H^{(X,\vec\varphi)}_k, H^{(Y,\vec\psi)}_k\right)=\epsilon_1$ and $d_T\left( H^{(Y,\vec\psi)}_k, H^{(Z,\vec\chi)}_k\right)=\epsilon_2$. Let $\eta$ be an arbitrarily small positive real number. The definition of $d_T$ implies that for each $(\vec u,\vec v)\in \Delta^+$ a surjective homomorphism $f$ from a subgroup $F$ of $ H^{(X,\vec\varphi)}_k(\vec u,\vec v)$ onto $ H^{(Y,\vec\psi)}_k(\vec u-(\vec \epsilon_1+\vec\eta),\vec v+(\vec \epsilon_1+\vec\eta))$, and a surjective homomorphism $g$ from a subgroup $G$ of $ H^{(Y,\vec\psi)}_k(\vec u-(\vec \epsilon_1+\vec\eta),\vec v+(\vec \epsilon_1+\vec\eta))$ onto $ H^{(Z,\vec\chi)}_k(\vec u-(\vec \epsilon_1+\vec\eta)-(\vec \epsilon_2+\vec\eta),\vec v+(\vec \epsilon_1+\vec\eta)+(\vec \epsilon_2+\vec\eta))$ exist. Now, we can consider the subgroup $\tilde{F}=f^{-1}(G)\subseteq  H^{(X,\vec\varphi)}_k(\vec u,\vec v)$. Obviously, the function $g_{|G}\circ f_{|\tilde{F}}$ is a surjective homomorphism from the subgroup $\tilde{F}$ of $ H^{(X,\vec\varphi)}_k(\vec u,\vec v)$ onto $ H^{(Z,\vec\chi)}_k(\vec u-(\vec \epsilon_1+\vec\eta)-(\vec \epsilon_2+\vec\eta),\vec v+(\vec \epsilon_1+\vec\eta)+(\vec \epsilon_2+\vec\eta))$.
    Analogously, we can construct a surjective homomorphism from a subgroup of $ H^{(Z,\vec\chi)}_k(\vec u,\vec v)$ onto $ H^{(X,\vec\varphi)}_k(\vec u-(\vec \epsilon_1+\vec\eta)-(\vec \epsilon_2+\vec\eta),\vec v+(\vec \epsilon_1+\vec\eta)+(\vec \epsilon_2+\vec\eta))$.
    This implies that $d_T\left( H^{(X,\vec\varphi)}_k, H^{(Z,\vec\chi)}_k\right)\le\epsilon_1+\epsilon_2+2\eta$. Since $\eta$ can be taken arbitrarily small, it follows that $d_T\left( H^{(X,\vec\varphi)}_k, H^{(Z,\vec\chi)}_k\right)\le\epsilon_1+\epsilon_2$, so proving the triangle inequality for $d_T$.
\end{description}
\end{proof}

As a simple example, we can easily check that $d_T$ takes the value $1$ between the first two persistent homology groups represented in Figure~\ref{groups}, and the value $\infty$ between the first and the third and between the second and the third persistent homology groups represented in the same figure.

\begin{remark}\label{rem_pd}
We observe that the pseudo-distance $d_T$ is not a distance. For example, we can consider the two Abelian groups $\bigoplus_{i=1}^\infty \mathbb{Z}$ (obtained by adding infinite copies of the integer numbers) and  $\mathbb{Z}_2\oplus\bigoplus_{i=1}^\infty \mathbb{Z}$. We can find two topological spaces $X$, $Y$ such that their homology groups in degree $1$ are these two groups, respectively. If we take the two filtering functions $\varphi:X\to\R$, $\psi:Y\to\R$ with $\varphi\equiv\psi\equiv 0$, we have that $H^{(X,\varphi)}_k(u,v)=\bigoplus_{i=1}^\infty \mathbb{Z}$ and $H^{(Y,\psi)}_k(u,v)= \mathbb{Z}_2\oplus\bigoplus_{i=1}^\infty \mathbb{Z}$ for any $(u,v)\in \Delta^+$ with $u\ge 0$. Obviously, if $(u,v)\in \Delta^+$ with $u< 0$ we get that both $H^{(X,\varphi)}_k(u,v)$ and $H^{(Y,\psi)}_k(u,v)$ are the trivial group. In any case, it is easy to find a surjective homomorphism from $ H^{(X,\varphi)}_k(u,v)$ onto $ H^{(Y,\psi)}_k(u,v)$ and a surjective homomorphism from $H^{(Y,\psi)}_k(u,v)$ onto $ H^{(X,\varphi)}_k(u,v)$, so that $d_T\left( H^{(X,\varphi)}_k, H^{(Y,\psi)}_k\right)=0$. However, the groups $H^{(X,\varphi)}_k(u,v), H^{(Y,\psi)}_k(u,v)$ are not isomorphic, for $(u,v)\in \Delta^+$ with $u\ge 0$.
\end{remark}

\subsection{Relationship between $d_T$ and $d_{match}$ in the $1$-dimensional case, if there is no torsion}

Let us go back to the case of filtering functions taking values in $\R$. In this case, if the Abelian group $G$ of homology coefficients is also a field, it is well known that
the persistent homology groups can be described by persistence diagrams, under the assumption that the topological space is triangulable (cf. \cite{CoEdHa07} for tame filtering functions, and \cite{CeDiFeFrLa09} for continuous filtering functions). The distance that is usually used to compare persistence diagrams is the matching distance $d_{match}$.

In plain words, the persistence diagram is the collection of all points $(u,v)$ where $u$ and $v$ represent the ``birth'' and ``death'' of a homology class, varying the sub-level sets of the topological space $X$ with respect to the filtering function $\varphi:X\to \R$. Each one of these pairs $(u,v)$ is endowed with a multiplicity. By definition, all the pairs $(u,u)$ belong to the persistent diagram, each of them counted with infinite multiplicity. The matching distance $d_{match}$ between two persistence diagrams is the infimum of the cost of a matching (i.e. a bijective correspondence) between two persistence diagrams. The cost of a matching $f$ is defined to be the maximum displacement of the points of the persistence diagrams, induced by $f$. Each displacement is measured by the max-norm on the real plane.

For a formal and precise definition of persistence diagram and $d_{match}$, we refer the interested reader to \cite{CoEdHa07,CeDiFeFrLa09}.

It is natural to wonder if there is any relationship between $d_T$ and $d_{match}$, under the hypotheses that we have just made.

We can prove the following result.

\begin{theorem}\label{dtlessthandelta}
Let $X,Y$ be two triangulable spaces endowed with
two continuous functions $\varphi: X \to \R$, $\psi: Y
\to \R$. Let us consider the persistent homology groups $H^{(X,\varphi)}_k$, $H^{(Y,\varphi)}_k$ with homology coefficients in a field.
Let $D_\varphi$ and $D_\psi$ be the persistent diagrams associated with the filtering functions $\varphi$ and $\psi$, respectively.
Then, for any integer $k$, $d_T\left( H^{(X,\varphi)}_k, H^{(Y,\psi)}_k\right)\le d_{match}(D_\varphi,D_\psi)$.
\end{theorem}

\begin{proof} Since the coefficients are taken in a field, our persistent homology groups have no torsion, and they can be described by their persistent Betti numbers functions $\beta_\varphi$, $\beta_\psi$, where we set
$\beta_\varphi(u,v)=\dim {H}_k^{(X,\varphi)}(u, v)$, $\beta_\psi(u,v)=\dim {H}_k^{(Y,\psi)}(u, v)$, for $(u,v)\in \Delta^+$.
Obviously, for each $k\in\mathbb{Z}$, we have different persistent Betti numbers functions
$\beta_\varphi$ of $\varphi$ (which should be denoted
$\beta_{\varphi,k}$, say) but, for the sake of notational
simplicity,  we omit adding any reference to $k$.

Let us assume that the matching distance $d_{match}$ between our persistence diagrams takes the value $\bar \lambda$. Then, from the Representation Theorem 2.17 in \cite{CeDiFeFrLa09} (cf. also the \emph{$k$-triangle Lemma} in \cite{CoEdHa07}), it follows that for every $\lambda>\bar\lambda$ and every pair $(u,v)\in \Delta^+$ the inequalities $\beta_\varphi(u-\lambda,v+ \lambda)\leq\beta_\psi(u,v)$ and $\beta_\psi(u-\lambda,v+\lambda)\leq\beta_\varphi(u,v)$ hold. This fact immediately implies that for each $(u,v)\in \Delta^+$
\begin{enumerate}
\item a surjective homomorphism from a subgroup of $ H^{(X,\varphi)}_k(u,v)$ onto $ H^{(Y,\psi)}_k(u-\lambda,v+\lambda)\}$ exists;
\item a surjective homomorphism from a subgroup of $ H^{(Y,\psi)}_k(u,v)$ onto $ H^{(X,\varphi)}_k(u-\lambda,v+\lambda)\}$ exists.
\end{enumerate}
Hence $d_T\left( H^{(X,\varphi)}_k, H^{(Y,\psi)}_k\right)\le\lambda$ for any $\lambda>\bar\lambda$. This implies our thesis.
 \end{proof}

Theorem~\ref{dtlessthandelta} leaves open the possibility that $d_T$ and $d_{match}$ coincide when the persistent homology groups have no torsion and hence $d_{match}$ is defined. This does not happen. Indeed, for any natural number $m\ge 2$ we can give an example for which
$d_T\left( H^{(X,\varphi)}_k, H^{(Y,\psi)}_k\right)\le \frac{2}{m}\cdot d_{match}(D_\varphi,D_\psi)$:

\begin{example}\label{ex1}
Let us consider a natural number $m\ge 2$. It is easy to define two triangulable spaces endowed with filtering functions whose persistent diagrams consist of the points  $(0,\infty),(0,m),(1,m+1),\ldots,(i,m+i),\ldots,(m,2m)$ and $(0,\infty),(0,m),(1,m+1),\ldots,(i,m+i),\ldots,(m-1,2m-1)$, respectively. Here all the points have multiplicity one. In Figure \ref{dtdelta} we can see the case $m=4$. The values of the corresponding persistent Betti number functions are displayed in Figure \ref{dtdeltavalues}. It is easy to check that in this case $d_T$ and $d_{match}$ take the value $1$ and $\frac{m}{2}$, respectively.
\end{example}

\begin{figure}[h]
\begin{center}
\includegraphics[width=14cm]{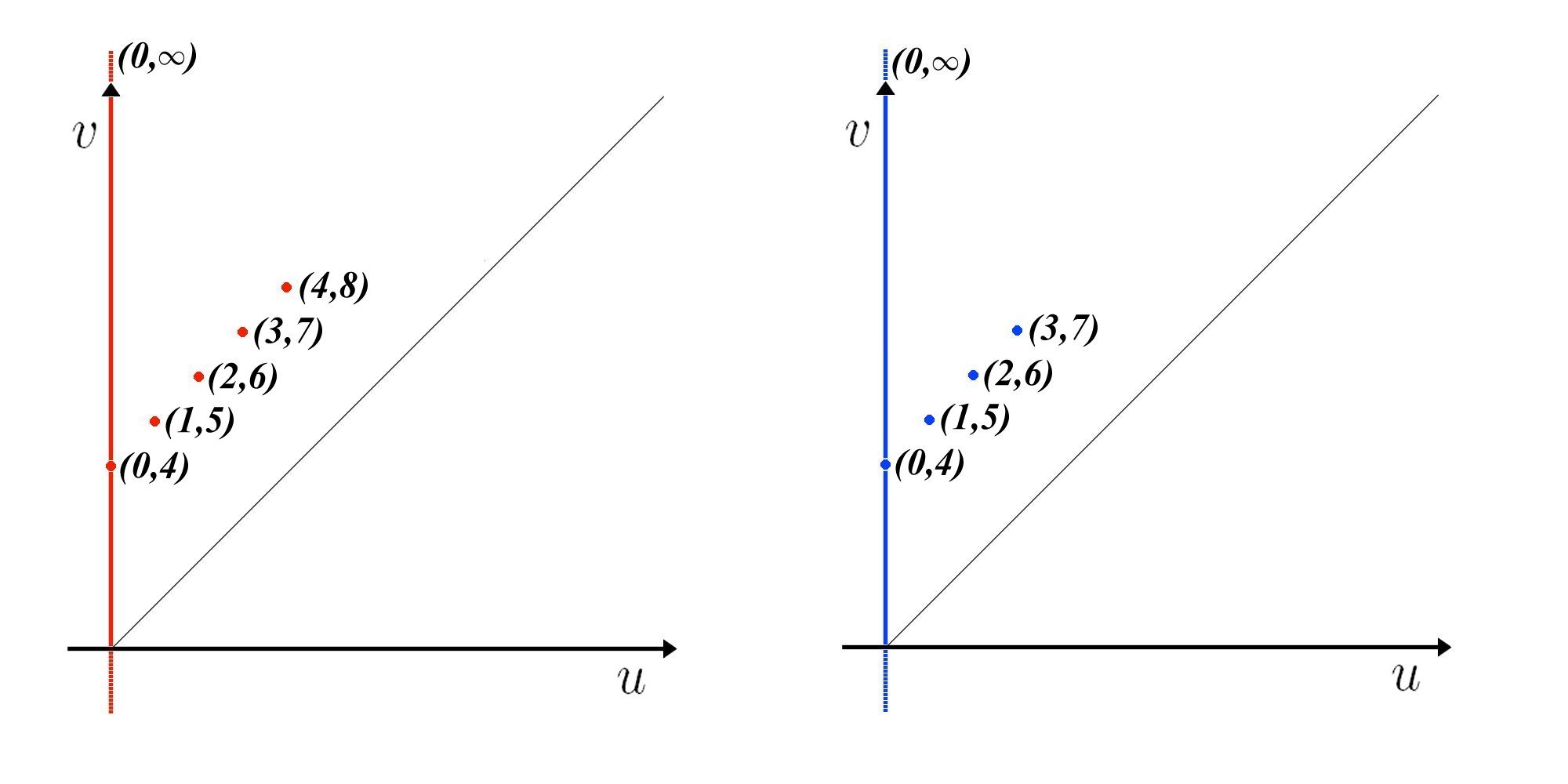}
\caption{It is easy to check that the matching distance $d_{match}$ between the two persistent diagrams displayed in this figure equals $2$.}\label{dtdelta}
\end{center}
\end{figure}

\begin{figure}[h]
\begin{center}
\includegraphics[width=14cm]{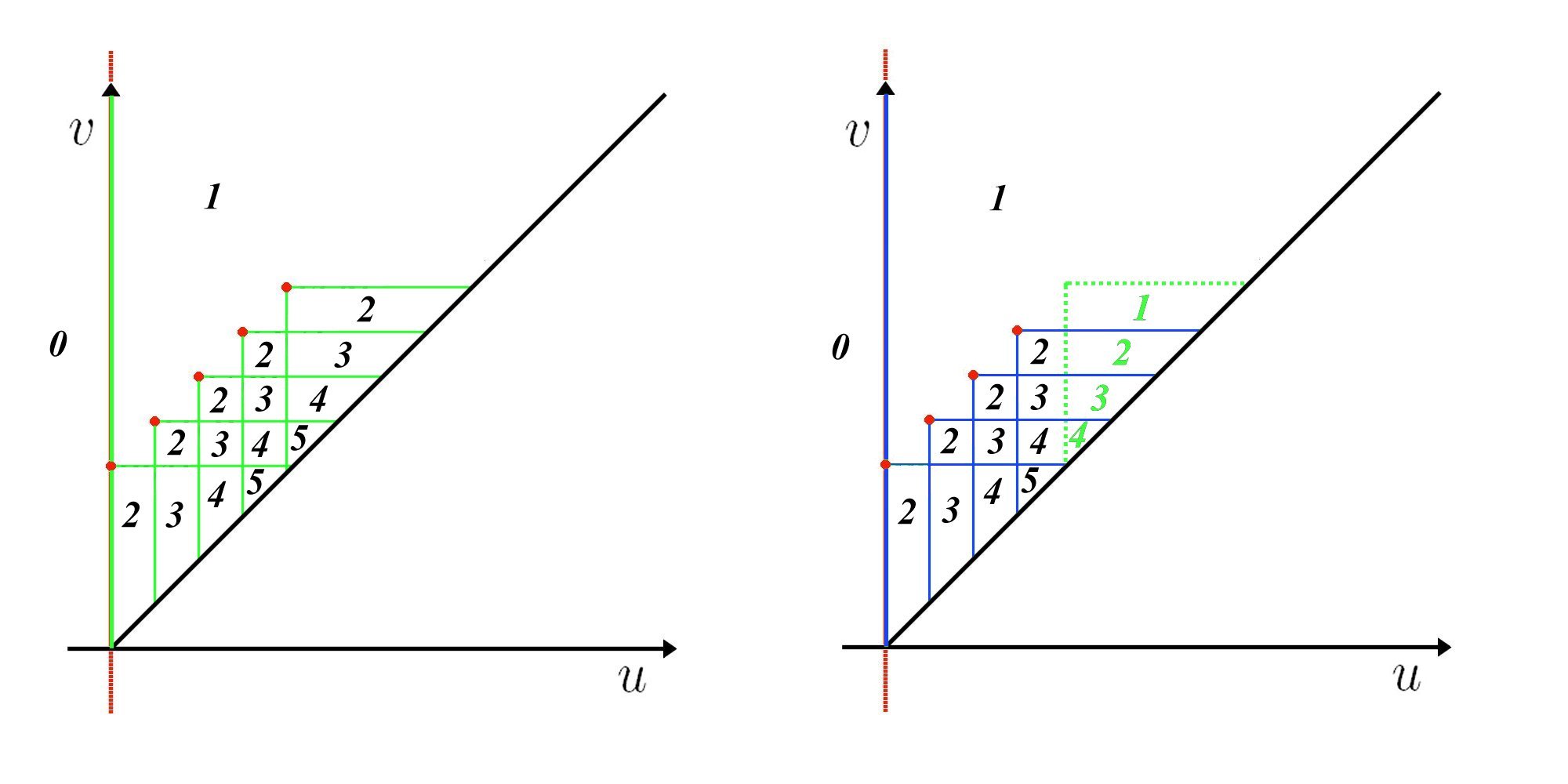}
\caption{Values of the persistent Betti numbers functions represented by the persistent diagrams in Figure \ref{dtdelta}. In the picture on the right, the dotted triangle represents the part of domain where the persistent Betti numbers functions differ from each other. It is easy to check that the pseudo-distance $d_T$ between the persistent homology groups described by the two persistent diagrams displayed in Figure~\ref{dtdelta} equals $1$.}\label{dtdeltavalues}
\end{center}
\end{figure}

Theorem~\ref{dtlessthandelta} and Example~\ref{ex1} show that the pseudo-distance $d_T$ is strictly less informative than the matching distance $d_{match}$.
However, we recall that $d_T$ can be applied also in presence of torsion, while $d_{match}$ cannot.

\subsection{Stability of the pseudo-distance $d_T$}
Another relevant reason to study the pseudo-distance $d_T$
is its stability. This stability can be expressed as a lower bound
for the natural pseudo-distance $\delta$.
Before proceeding, let us recall the definition of $\delta$.

\begin{definition}[Natural pseudo-distance]
Let $X,Y$ be two compact topological spaces  endowed with
two continuous functions $\vec \varphi: X \to \R^n$, $\vec \psi: Y
\to \R^n$. The \emph{natural pseudo-distance} between the pairs $(X,\vec
\varphi)$ and $(Y,\vec \psi)$, denoted by
$\delta\left((X,\vec\varphi),(Y,\vec\psi)\right)$,  is
\begin{itemize}
\item[(i)] the number $\inf_{h}\|\vec
\varphi-\vec \psi \circ h\|_\infty$ where $h$ varies in the set
of all the homeomorphisms between $X$ and $Y$, if $X$ and
$Y$ are homeomorphic;
\item[(ii)] $\infty$, if $X$ and $Y$ are
not homeomorphic.
\end{itemize}
\end{definition}

We point out that the natural pseudo-distance is not a distance
because it can vanish on two distinct pairs. However, it is
symmetric, satisfies the triangular inequality, and vanishes on
two equal pairs.

The natural pseudo-distance has been studied in
\cite{DoFr04,DoFr07,DoFr09} in the case of scalar-valued filtering
functions on manifolds, and in \cite{FrMu99} in the case of
vector-valued filtering functions on manifolds.

We are now ready to show that the pseudo-distance $d_T$ is stable.
We begin by proving a useful lemma.

\begin{lemma}\label{lemmahom}
Let $X,Y$ be two homeomorphic compact topological spaces  endowed with
two continuous functions $\vec \varphi: X \to \R^n$, $\vec \psi: Y
\to \R^n$.
Assume that $h:X\to Y$ is a homeomorphism such that $\|\vec\varphi-\vec\psi\circ h\|_\infty\le \epsilon$. Then for every pair $(\vec u,\vec v)\in \Delta^+$ a subgroup $S$ of $H^{(X,\vec\varphi)}_k(\vec u,\vec v)$ and a surjective homomorphism $f:S\to  H^{(Y,\vec\psi)}_k(\vec u-\vec \epsilon,\vec v+\vec \epsilon)$ exist.
\end{lemma}

\begin{proof}\label{prooflemmahom}
The homeomorphism $h$ induces a chain isomorphism $h_*$. Obviously, $h_*$ takes cycles to cycles and boundaries to boundaries.  Let us consider the set $S$ of the elements of $ H^{(X,\vec\varphi)}_k(\vec u,\vec v)$ that contain at least one cycle whose image by $h_*$ is a cycle in $Y\langle\vec\psi\preceq \vec u-\vec\epsilon\rangle$. We claim that $S$ is a subgroup of $ H^{(X,\vec\varphi)}_k(\vec u,\vec v)$. Indeed, if $a_1,a_2\in S$ then two cycles $\gamma_1\in a_1$, $\gamma_2\in a_2$ exist, such that $h_*(\gamma_1), h_*(\gamma_2)$ are cycles in $Y\langle\vec\psi\preceq \vec u-\vec\epsilon\rangle$. Hence the cycle $h_*(\gamma_1)+h_*(\gamma_2)=h_*(\gamma_1+\gamma_2)$ is a cycle in $Y\langle\vec\psi\preceq \vec u-\vec\epsilon\rangle$, too. Since $\gamma_1+\gamma_2\in a_1+a_2$, it follows that the equivalence class $a_1+a_2$ belongs to $S$. That means that $S$ is closed with respect to the sum.

Furthermore,
if $a\in S$ then a cycle $\gamma\in a$ exists, such that $h_*(\gamma)$ is a cycle in $Y\langle\vec\psi\preceq \vec u-\vec\epsilon\rangle$. Hence the cycle $h_*(-\gamma)=-h_*(\gamma)$ is a cycle in $Y\langle\vec\psi\preceq \vec u-\vec\epsilon\rangle$, too. Since $-\gamma\in -a$, it follows that the equivalence class $-a$ belongs to $S$. That means that $S$ is closed with respect to the computation of the opposite.

In summary, $S$ is a subgroup of $ H^{(X,\vec\varphi)}_k(\vec u,\vec v)$.

Now, let us define the homomorphism $f:S\to  H^{(Y,\vec\psi)}_k(\vec u-\vec \epsilon,\vec v+\vec \epsilon)$. In the following, for each $a\in S$, let us fix a cycle $\gamma_a\in a$ such that $h_*(\gamma_a)$ is a cycle in $Y\langle\vec\psi\preceq \vec u-\vec \epsilon\rangle$. For every $a\in S$ we define $f(a)$ as the element of $ H^{(Y,\vec\psi)}_k(\vec u-\vec \epsilon,\vec v+\vec \epsilon)$ that contains $h_*(\gamma_a)$ (i.e., in symbols, $[h_*(\gamma_a)]$).

We claim that $f$ is a homomorphism. In order to do that, let us consider three elements $a_1,a_2,a\in  H^{(X,\vec\varphi)}_k(u,v)$, such that $a_1+a_2=a$. Obviously,
$$f(a_1)+f(a_2)=[h_*(\gamma_{a_1})]+[h_*(\gamma_{a_2})]=[h_*(\gamma_{a_1})+h_*(\gamma_{a_2})]=[h_*(\gamma_{a_1}+\gamma_{a_2})]$$
$$f(a)=[h_*(\gamma_{a})].$$
Now, $[h_*(\gamma_{a_1}+\gamma_{a_2})]-[h_*(\gamma_{a})]=[h_*(\gamma_{a_1}+\gamma_{a_2})-h_*(\gamma_{a})]=[h_*(\gamma_{a_1}+\gamma_{a_2}-\gamma_{a})]$.
Since $a_1+a_2=a$, a $(k+1)$-cycle $\tau$ in $X\langle\vec\varphi\preceq \vec v\rangle$ exists, such that $\partial \tau=\gamma_{a_1}+\gamma_{a_2}-\gamma_a$. Given that $\|\vec\varphi-\vec\psi\circ h\|_\infty\le \epsilon$, $h_*(\tau)$ is a $(k+1)$-cycle in $Y\langle\vec\psi\preceq \vec v+\vec\epsilon\rangle$. Moreover, recalling that $h_*$ is a chain map, $\partial h_*(\tau)=h_*(\partial \tau)=h_*(\gamma_{a_1}+\gamma_{a_2}-\gamma_a)$, so that $h_*(\gamma_{a_1}+\gamma_{a_2}-\gamma_a)$ is a boundary in $Y\langle\vec\psi\preceq \vec v+\vec\epsilon\rangle$ and hence $[h_*(\gamma_{a_1}+\gamma_{a_2}-\gamma_a)]$ is the null element in $ H^{(Y,\vec\psi)}_k(u-\epsilon,v+\epsilon)$. This proves that $f(a_1)+f(a_2)=f(a)$.

Finally, we claim that $f$ is surjective. Indeed, let  $b\in  H^{(Y,\vec\psi)}_k(u-\epsilon,v+\epsilon)$ and $\beta\in b$, so that $\beta$ is a cycle in $Y\langle\vec\psi\preceq \vec u-\vec\epsilon\rangle$. Let $\alpha$ be the cycle $\left((h_*)^{-1}\right)(\beta)=\left(h^{-1}\right)^*(\beta)$. Given that $\|\vec\varphi\circ h^{-1}-\vec\psi\|_\infty=\|\vec\varphi-\vec\psi\circ h\|_\infty\le \epsilon$, $\alpha$ is a cycle in $X\langle\vec\varphi\preceq \vec u\rangle$. Therefore the element $a\in  H^{(X,\vec\varphi)}_k(u,v)$ the cycle $\alpha$ belongs to is an element of the set $S$. Then $f(a)=[h_*(\gamma_{a})]$. Now, $\gamma_{a}$ and $\alpha$ are homologous in $X\langle\vec\varphi\preceq \vec v\rangle$. Once again because of the inequality $\|\vec\varphi-\vec\psi\circ h\|_\infty\le \epsilon$, it follows that $h_*(\gamma_{a})$ and $\beta=h_*(\alpha)$ are homologous in $Y\langle\vec\psi\preceq \vec v+\vec\epsilon\rangle$. Hence $f(a)=b$, and $f$ is proven to be surjective.

\end{proof}

Now we can prove that $d_T$ gives a lower bound for the natural pseudo-distance $\delta$.

\begin{theorem}\label{bound}
Let $X,Y$ be two compact topological spaces  endowed with
two continuous functions $\vec \varphi: X \to \R^n$, $\vec \psi: Y
\to \R^n$. Then
$d_T\left( H^{(X,\vec\varphi)}_k, H^{(Y,\vec\psi)}_k\right)\le \delta\left((X,\vec\p),(Y,\vec\psi)\right)$.
\end{theorem}

\begin{proof}\label{proofbound}
Let us set $\bar \delta=\delta\left((X,\vec\varphi),(Y,\vec\psi)\right)$.
Because of the definition of natural pseudodistance, for each $\eta>0$ we can find a homeomorphism
$h_\eta:X\to Y$ such that $\|\vec \varphi-\vec \psi\circ h_\eta\|_\infty\le \bar\delta+\eta$.
By applying Lemma \ref{lemmahom} to both the homeomorphisms
$h_\eta:X\to Y$ and $h_\eta^{-1}:Y\to X$, we get that $d_T\left( H^{(X,\vec\varphi)}_k, H^{(Y,\vec\psi)}_k\right)\le \bar\delta+\eta$.
Since $\eta$ can be chosen arbitrarily close to $0$, our thesis follows.
\end{proof}

Previous Theorem~\ref{bound} gives an easy way to get lower bounds for the natural pseudo-distance $\delta$:

\begin{corollary}\label{cor1}
Let $X,Y$ be two compact topological spaces  endowed with
two continuous functions $\vec \varphi: X \to \R^n$, $\vec \psi: Y
\to \R^n$. If no surjective homomorphism exists from a subgroup of $ H^{(X,\vec\varphi)}_k(\vec u,\vec v)$ onto $ H^{(Y,\vec\psi)}_k(\vec u',\vec v')\}$, then  $\delta\left((X,\vec\varphi),(Y,\vec\psi)\right)\ge \min_i\min \{u_i-u'_i,v'_i-v_i\}$.
\end{corollary}

\begin{proof}\label{proofcor1}
It follows from Theorem~\ref{bound}, observing that our hypotheses easily imply that the inequality $d_T\left( H^{(X,\vec\varphi)}_k, H^{(Y,\vec\psi)}_k\right)\ge \min_i\min \{u_i-u'_i,v'_i-v_i\}$ holds.
\end{proof}

Our last result illustrates the relationship between $d_T$ and the change of the filtering function, measured by the max-norm. It shows that $d_T$ is a stable distance.

\begin{corollary}\label{cor2}
Let $X$ be a compact topological space  endowed with
two continuous functions $\vec \varphi: X \to \R^n$, $\vec \psi: X
\to \R^n$. Then $d_T\left( H^{(X,\vec\varphi)}_k, H^{(X,\vec\psi)}_k\right)\le \|\vec \p-\vec\psi\|_\infty$.
\end{corollary}

\begin{proof}\label{proofcor2}
It follows from Theorem~\ref{bound}, observing that $\delta\left((X,\vec\p),(X,\vec\psi)\right)\le \|\vec \p-\vec\psi\|_\infty$.
\end{proof}

\section*{Conclusions}
The main contribution of this paper is the proof that it is possible to compare persistent homology groups in a stable way also in presence of torsion, although no complete representation by persistent diagrams or other compact descriptors is available, at the time we are writing. Our approach, based on the new pseudo-distance $d_T$, opens the way to experimentation in the case that the homology coefficients are taken in an Abelian group instead of a field. In this case, the classical matching distance between persistent diagrams cannot be applied without forgetting the information contained in the torsion of the persistent homology groups. We plan to extend these ideas to other concepts and problems in Multidimensional Persistence.
\bigskip

{\bf Acknowledgments.} Research partially supported by DISTEF. The author thanks Sara, and all his friends in the ``Coniglietti'' Band.

\end{document}